\documentclass[12pt,a4paper]{amsart}
\usepackage[margin=2.2cm]{geometry}
\usepackage{amssymb,xspace}
\usepackage[hidelinks]{hyperref}
\numberwithin{equation}{section}
\theoremstyle{definition}
\newtheorem{definition}{Definition}[subsection]
\newtheorem{example}[definition]{Example}
\newtheorem{notation}[definition]{Notation}
\newtheorem{remark}[definition]{Remark}
\theoremstyle{plain}
\newtheorem{theorem}[definition]{Theorem}
\newtheorem{proposition}[definition]{Proposition}
\newtheorem{prop}[definition]{Proposition}
\newtheorem{lemma}[definition]{Lemma}
\newtheorem{corollary}[definition]{Corollary}

\usepackage[T2A,T1]{fontenc}\DeclareSymbolFont{cyrillic}{T2A}{cmr}{m}{n}
\DeclareMathSymbol{\Sha}{\mathalpha}{cyrillic}{216}
\newcommand{\sha}{\mathop{\scriptstyle\Sha}\nolimits}
\newcommand{\osha}{\widetilde\sha}
\renewcommand{\leq}{\leqslant}
\renewcommand{\le}{\leqslant}
\renewcommand{\geq}{\geqslant}
\renewcommand{\ge}{\geqslant}

\def\compositions#1#2{C_{\!#2}^{#1}}

\def\twocompositions#1#2#3{\compositions{#1}{#3}\times\compositions{#2}{#3}}
\usepackage{tikz-cd}

\def\ainfty{\texorpdfstring{\ensuremath{A_\infty}\xspace}{A∞}}
\def\binfty{\texorpdfstring{\ensuremath{B_\infty}\xspace}{B∞}}

\def\pdot{\bullet}

\makeatletter
 
\DeclareRobustCommand{\qdot}{\mathbin{\mathpalette\qdot@\relax}}
 \newcommand{\qdot@}[2]{%
   \ooalign{%
     $\m@th#1\circ$\cr
     \hidewidth$\m@th#1\cdot$\hidewidth\cr
   }%
 }
 \makeatother

\def\deg#1{\ensuremath\left| #1\right|}

\def\NN{{\mathbb{N}}}

\def\K{{\mathbb K}}

\def\P{\operatorname{Prim}}
\def\As{\mathcal A\mathit s}
\def\dAs{\mathit d\As}
\def\sgn{\textrm{sgn}}
\def\im{\textrm{im}}
\def\inv{\textrm{inv}}
\def\id{\operatorname{id}}

\def\myclap#1{\makebox[0em]{$\scriptstyle #1$}}

\newcommand{\overbar}[1]{\mkern 1.5mu\overline{\mkern-1.5mu#1\mkern-1.5mu}\mkern 1.5mu}

\author[I.\ G\'alvez]{Imma G\'alvez-Carrillo}
\address{Universidad de M\'alaga, IMTECH-UPC, and Centre de Recerca Matemàtica
}
\email{imma.galvez@uma.es}
\author[M.\ Ronco]{María Ronco}
\address{IMAFI, Universidad de Talca
Campus Norte,
Chile
  }
\email{mronco@utalca.cl}
\author[A.\ Tonks]{Andy Tonks}
\address{Universidad de Málaga, Spain}
\email{at@uma.es}

\title{On differential Hopf algebras and $B_{\infty}$ algebras}

\begin{document}
\maketitle
\tableofcontents

\section*{Introduction}

We establish a structure theorem,  analogous to the classical result of Milnor and Moore, for differential graded Hopf algebras and $B_\infty$ algebras. This extends the work of Loday and the second author \cite{LodayRonco2006} where
the  ungraded non-differential case was treated: it was shown that any cofree Hopf algebra carries an underlying multibrace structure that restricts to the subspace of primitives, and conversely that any cofree Hopf algebra may be reconstructed from a multibrace algebra via a universal enveloping 2-associative algebra. 

In particular, in the case of differential graded Hopf algebras,  we establish an extension of the Loday--Ronco multibrace structure to a $B_{\infty}$ structure.

$B_\infty$ algebras first appeared in algebraic topology, implicitly, in work of Baues \cite{Baues} and play a central role in the study of the algebraic structure carried by the Hochschild cochain complex and its singular or Tate analogue \cite{BriggsRubio, ChenLiWang, Keller, LowenBergh,Wang}.
The explicit definition of a $B_\infty$ structure, on a vector space or a chain complex $A$, was first formulated in \cite[Section 5.2]{GeJo} as a differential $d$ and a product $\mu$ on the bar construction $(BA,\Delta)$ making it a differential bialgebra.
Unpacking this definition, a $B_\infty$ algebra is equipped with families of multilinear structure maps $(m_{n})_{n\ge1}$ and $(m_{i,j})_{i,j\ge1}$ of arities $n$ and $i+j$ satisfying certain `strong homotopy algebra'  relations.
This includes the notion of $A_{\infty}$ algebra, where we just require a differential coalgebra $(BA,\Delta,d)$, that is, the structure maps $m_{n}$.
It also includes the notion of multibrace algebra, where we just require a bialgebra $(BA,\Delta,\mu)$, that is, the structure maps $m_{i,j}$. Multibrace algebras were termed non-differential $\mathbf B_{\infty}$ algebras  in  \cite{LodayRonco2006} and $B_\infty$ algebras in \cite{patras}.

Our second aim is provide further insight into a construction of Börjeson~\cite{Kai}, who showed that any associative algebra with a square zero endomorphism (not required to be a derivation) carries an underlying $A_{\infty}$ structure. Markl in \cite{markl} explained the origin of this $A_{\infty}$ structure as a twisting (arising from the multiplication) of the given trivial $A_{\infty}$ algebra.

As we explain below, Börjeson's $A_\infty$ structure is just the right adjoint to the tensor (co)algebra functor. Another crucial observation is that the Loday--Ronco construction fits into this framework: we discover that the underlying multibrace structure carried by any 2-associative algebra, as defined in~\cite[Section 3.3 and Proposition 3.4]{LodayRonco2006},
originates from a twisting of a quasi-trivial multibrace algebra (a non-commutative quasi-shuffle algebra). See Example~\ref{ex:lr-mb} and Theorem~\ref{thm:lr-mb} below for details.

To keep the presentation self-contained and accessible we recall much material which should be known to the expert reader.
To simplify notation and remain consistent with \cite{Kai,LodayRonco2006,markl}
we choose to work with the tensor coalgebra $T^c(A)$ instead of the bar construction $BA=T^c(A[1])$.
Thus our definitions of $A_\infty$ and  $B_\infty$ differ from the classical ones mentioned above. The translation (of the degrees and signs involved) is straightforward.

\subsection*{Outline of results}
 
Suppose $(V,\partial,\pdot)$ is a differential graded algebra. Then (Proposition~\ref{prop:difstuf}) it has a quasi-trivial $B_\infty$ structure $d_{\partial}$, $\mu_{\pdot}$ with all structure maps zero except $m_{1}=\partial$ and $m_{1,1}=\pdot$.
%If the operation $\pdot$ is zero this is the shuffle algebra on $(V,\partial)$.
More generally, let $\dAs^{1,1}$ be the category of differential graded algebras  $(V,\partial,\pdot)$ with an a priori unrelated extra associative binary operation $\qdot$ on $V$. Now (Theorem~\ref{thm:lr-mb}) we have an underlying $B_{\infty}$ algebra  functor
\begin{align*}
  \dAs^{1,1}\text{-alg} &\to B_{\infty}\text{-alg} \tag{$\star$} \\
(V,\partial,\pdot,\qdot)&\mapsto (V,d_{\partial}^{\qdot},\mu_{\pdot}^{\qdot})
\end{align*}
obtained by twisting the quasi-trivial $B_{\infty}$ structure using the extra multiplication $\qdot$.
Furthermore
%Almost by definition of twisting
there is a natural projection map
\[
\varepsilon_{V}:(T^{c}(V),d_{\partial}^{\qdot},\mu_{\pdot}^{\qdot},\qdot)\to (V,\partial,\pdot,\qdot),
\]
which is a  $\dAs^{1,1}$ algebra homomorphism, where the extra multiplication $\qdot $ on  $T^{c}(V)$ is concatenation of tensors.

 Any differential graded Hopf algebra $(H,\partial,\pdot,\Delta)$ which is cofree, $H=T^c(A)$, gives rise to a $\dAs^{1,1}$ algebra $(T^c(A),\partial,\pdot,\qdot)$ where $\qdot$ is concatenation of tensors, and indeed to a $B_{\infty}$ algebra $A$. Hence $(\star)$ defines the underlying $B_\infty$ structure on $H$. By induction on the arity of the structure maps (Theorem~\ref{thm:incl-hm-b-inf}) we show the natural inclusion map
\[
\eta_{A}:A\to T^{c}(A)=H
\]
is a $B_{\infty}$ algebra homomorphism, and in this sense the underlying $B_{\infty}$ structure on $H$ can be regarded as restricting to the primitives.

An analogue of the classical universal enveloping algebra can be defined
as a certain quotient of the free $\dAs^{1,1}$ algebra:
\begin{align*}
  B_{\infty}\text{-alg} &\to
  \dAs^{1,1}\text{-alg}
                         \\
U(A)&=F_{\dAs^{1,1}}(A)/I.
\end{align*}
This provides a left adjoint to the underlying $B_\infty$ algebra.
By virtue of the natural transformations $(\eta_A)$, $(\varepsilon_V)$ above,
an alternative construction (Proposition \ref{prop:tc=u}) of the left adjoint is given by
\begin{align*}
  B_{\infty}\text{-alg} &\to
  \dAs^{1,1}\text{-alg}
                         \\
(A,d,\mu)&\mapsto (T^c(A),d,\mu,\qdot).
\end{align*}
We can rephrase this: our underlying $B_{\infty}$ structure $(\star)$ is just the right adjoint to this essentially tautologous tensor (co)algebra functor.

Thus $U(A)$ and $T^{c}(A)$ are canonically isomorphic $\dAs^{1,1}$ algebras and, taking into account the deconcatenation comultiplication, we can regard the universal enveloping algebra as a $\dAs^{1,1}$ \emph{bialgebra}, and also as a cofree differential Hopf algebra. Now given a conilpotent $\dAs^{1,1}$ bialgebra we show (Theorem~\ref{thm:milnor-moore}) that the underlying $B_{\infty}$ structure restricts to the space of primitives and gives an adjoint equivalence of categories
\[\begin{tikzcd}[every arrow/.append style={shift left=0.7ex}]
 B_{\infty}\text{-alg} \ar[r,"U"]&\ar[l,"\P"]
  \dAs^{1,1}\text{-bialg}_{\text{conil}}\,.
  \end{tikzcd}
\]
That is, every $B_{\infty}$ algebra arises as the primitives of a $\dAs^{1,1}$ bialgebra.

In particular any conilpotent $\dAs^{1,1}$ bialgebra is cofree, and is essentially a cofree differential Hopf algebra.

\subsection*{Acknowledgments}
This work received support from a Spanish university requalification and mobility 
                grant (UP2021-034, UNI/551/2021) with NextGenerationEU funds, from research grants PID2019-103849GB-I00
        and
        PID2020-117971GB-C22
        (AEI/FEDER, UE) of Spain, and
                2021-SGR-0603  of Catalonia, and through 
                the Severo Ochoa and Mar\'ia de Maeztu Program for Centers and Units 
                of Excellence in R\&D (CEX2020-001084-M).

We would like to thank Vladimir Dotsenko and Martin Markl for helpful discussions, and the Institut de Matemàtiques, Universitat de Barcelona, and the Max-Planck-Institut Bonn, for their hospitality at different stages of this project.

\section{Coalgebras, algebras and bialgebras}
\subsection{Grading and Koszul signs}
We work with graded vector spaces $V$ over a field $\K$, and denote the degree of a homogeneous element $v\in V$ by $\deg v$.
The tensor product of graded spaces is graded by $\deg{a\otimes b}=\deg a+\deg b$. The tensor product of (homogeneous) graded maps is given by
\begin{align*}
(f\otimes g)(a\otimes b)&=(-1)^{\deg g\deg a}f(a)\otimes g(b),\label{eq:f tensor g}
\intertext{which implies}
(f\otimes g)\circ(h\otimes k)&=(-1)^{\deg g\deg h}(f\circ h)\otimes (g\circ k).
\end{align*}
%If not stated otherwise, linear maps are assumed to be of degree zero.
These signs arise from\footnote{The symmetry is used to define the evaluation map $\text{ev}\colon\!\hom(V_1,W_1)\!\otimes\!\hom(V_2,W_2)\,\otimes\,V_1\!\otimes \!V_2\to W_1\!\otimes\! W_2$.} the natural symmetry isomorphism of the tensor product
\[\sigma\colon V\otimes W\cong W\otimes V,\qquad\sigma(a\otimes b)=(-1)^{\deg a\deg b}b\otimes a
  .
\]

Given a permutation $\sigma\in\Sigma_{n}$ we may denote by the same name the symmetry isomorphism of an $n$-fold tensor product that moves the $i$-th factor to the $\sigma(i)$-th position,
\begin{align*}
  V_{1}\otimes\cdots\otimes V_{n}
  &\quad \stackrel \sigma\longrightarrow\quad
    V_{\sigma^{-1}(1)}\otimes\cdots\otimes V_{\sigma^{-1}(n)},\\
  \sigma(v_1\otimes \dots\otimes v_n&)= (-1)^\kappa \;v_{\sigma^{-1}(1)}\otimes\dots\otimes v_{\sigma^{-1}(n)}.
\end{align*}
The formula for the Koszul signs arising in this symmetry isomorphism is
\begin{equation}
\label{eq:koszul-sgn}
\kappa\;\;=\;\;\sgn(v_1,\dots,v_n;\sigma)\;\;=\;\;\sum_{\myclap{(p,q)\in\inv(\sigma)}}
  \deg{v_p}
  \deg{v_q}.
\end{equation}
Here the sum is over inversions $(p,q)$ of $\sigma$, that is, all $p<q$ with $\sigma(p)>\sigma(q)$.

\subsection{Coalgebras}
A  \emph{coalgebra} (or more explicitly, a counital coassociative coalgebra over $\K$)
is a graded vector space $C$ equipped with a
comultiplication $\Delta\colon C\to C\otimes C$
and a
counit $\varepsilon\colon C\to \K$
which satisfy the coassociativity and counit laws
\begin{equation}
  (\Delta\otimes1)\Delta=(1\otimes\Delta)\Delta,\qquad
 \lambda(\varepsilon\otimes\id)\Delta=\id=\rho(\id\otimes\varepsilon)\Delta
.\end{equation}
The structure maps $\Delta,\varepsilon$ have degree zero, and
\begin{equation}\label{Ktrivial}
  \lambda:\K\otimes C\to C,\qquad \rho:C\otimes\K\to C
\end{equation}
  are the canonical isomorphisms which we often regard as identity maps and silently omit from the notation. % Ktrivial
A \emph{coalgebra map} is a linear map that preserves the comultiplication and counit.
We may write comultiplication using Sweedler notation without an explicit summation symbol,
\[
\Delta(x)\;\;=\;\;x_{(1)}\otimes x_{(2)}.
\]
The tensor product $C\otimes D$ of coalgebras is canonically a coalgebra,  with counit $\varepsilon(x\otimes y)=\varepsilon(x)\varepsilon(y)$ and comultiplication
\[\Delta\colon
  C\otimes D
  \xrightarrow{\Delta\otimes\Delta}
  (C\otimes C)\otimes (D\otimes D)
\xrightarrow{(23)}
(C\otimes D)\otimes (C\otimes D)
.
\]
As the two inner tensor factors are permuted, Koszul signs appear when applied to elements,
\[
\Delta(x\otimes y)\;\;=\;\;(-1)^{\deg{y_{(1)}}\deg{x_{(2)}}}(x_{(1)}\otimes y_{(1)})\otimes (x_{(2)}\otimes y_{(2)}).
\]

\subsection{Reduced coalgebras}

A \emph{unit} or \emph{coaugmentation} for a coalgebra $C$ is
a coalgebra map  $\eta\colon\K\to C$, or equivalently
a specified element $\eta(1_{\K})=1\in C$ which is group-like: $\Delta(1)=1\otimes1$. Then $\varepsilon(1)=1_{\K}$ and there is a splitting
\[
(\eta\varepsilon,J)\colon C\stackrel\cong\longrightarrow \K1\oplus \overbar C
\]
where $J=\id-\eta\varepsilon:C\to C$ and $\overbar C=\ker(\varepsilon)=\im(J)$.

For any unital coalgebra $(C,\Delta,\varepsilon,1)$  there is a well-defined coassociative operation
\begin{align*}
  &\overbar\Delta\colon\overbar C\to\overbar C\otimes\overbar C
\intertext{where, for $x\in\overbar C$,}
&\overbar\Delta(x)\;=\;\Delta(x) - x\otimes 1 - 1\otimes x
\;=\;J^{\otimes2}\Delta(x).
\end{align*}
The (non-unital, non-counital) coalgebra $(\overbar C,\overbar\Delta)$ is called the \emph{reduced} coalgebra.
This gives an equivalence between the categories of unital counital coalgebras and their maps and of non-unital non-counital algebras and their maps.

The kernel of $\overbar\Delta$ is the subspace $\P(C)$ of \emph{primitive} elements of $C$.

 \begin{remark}\label{rk:unit-law}
   Saying that a binary operation defined by a linear map $\mu\colon C\otimes C\to C$ is \emph{unital} may mean, according to context, either the condition $\mu(1\otimes 1)=1$ or the condition $\mu(1\otimes x)=\mu(x\otimes 1)=x$ for all $x\in C$.
   The latter condition will often be paraphrased as \emph{satisfies the unit law}. In this case specifying $\mu$ is equivalent to specifying the component $\overbar C\!\otimes\! \overbar C\to C$ as the remaining components
   are the canonical isomorphisms, cf. \eqref{Ktrivial},
\[\mu(\eta\otimes\id)%=\lambda Ktrivial
  \colon \K\!\otimes\! C
  \stackrel=\longrightarrow C,
  \qquad\qquad \mu(\id\otimes\eta)%=\rho Ktrivial
  \colon
C\!\otimes\!\K
\stackrel=\longrightarrow C.
\]
\end{remark}

\subsection{The primitive filtration}
Consider the filtration of a unital coalgebra $C$ defined in~\cite[Section 3]{Q} by
$F_0C=\K 1$ and
\[
F_rC\;=\;\K 1\oplus\{ x\in \overbar C\;|\; \overbar\Delta(x)\in F_{r-1}C\otimes F_{r-1}C\}
\qquad (r\geq1).\]
Denote by
$\Delta^{(r)}\colon C\to C^{\otimes (r+1)}$ and $\overbar\Delta^{(r)}\colon\overbar C\to\overbar C^{\otimes (r+1)}$
the $r$-fold iterations of the coassociative comultiplications $\Delta$ and $\overbar\Delta$. For $r=0$ these are the identity maps,  and $\Delta^{(-1)}=\varepsilon:C\to\K$.
\begin{definition}The \emph{conilpotent radical} $R(C)$ of a unital coalgebra $C$ is the subcoalgebra $\bigcup F_rC$.
\end{definition}
\begin{definition}[Connected and conilpotent coalgebras]
  A connected coalgebra 
  (in the sense of \cite{LodayRonco2006})
  is a unital coalgebra $C$ such that
  the above filtration is exhaustive, so that $R(C)=C$.
  A conilpotent coalgebra is a unital coalgebra $C$ such that
  for each $x\in\overbar C$ there exists $r\in\NN$ with $\overbar \Delta^{(r)}(x)=0$.
\end{definition}

It is well known that connected coalgebras are conilpotent: it was noted in \cite{LodayRonco2006} and can be seen in Quillen's proof of \cite[Proposition 4.1]{Q}. The converse was pointed out in \cite[Appendix B]{GaKaTo}. We have:
\begin{lemma}
  A unital coalgebra $C$ is conilpotent if and only if it is connected. In fact,
  \[ F_rC\;\;=\;\;\K1\oplus \ker(\overbar\Delta^{(r)})\;\;=\;\;\ker(J^{\otimes (r+1)}\Delta^{(r)})
  \]
and $R(C)$ is the largest conilpotent subcoalgebra of $C$.
\end{lemma}
\begin{proof}
  Since $J^{\otimes (r+1)}\Delta^{(r)}(1)=0$ and
  $J^{\otimes (r+1)}\Delta^{(r)}(x)
%    = ((\id-\eta\varepsilon)^{\otimes(r+1)}\Delta^{(r)})_{|\overbar{\scriptstyle C}}
  =\overbar\Delta^{(r)}(x)$ for $x\in\overbar C$ the second equality holds.
Now note that taking kernels and the tensor product commute in the following sense
\begin{align*}
\ker(\overbar\Delta^{(r)})\otimes \ker(\overbar\Delta^{(r)})
  &\;=\;\ker(\overbar\Delta^{(r)}\otimes \id)  \;\cap\; \ker(\id\otimes\overbar\Delta^{(r)}).
\intertext{The first equality therefore follows inductively from}
{\overbar\Delta}^{-1}\left(\ker(\overbar\Delta^{(r)})\otimes \ker(\overbar\Delta^{(r)})\right)
  &  \;=\;
\ker((\overbar\Delta^{(r)}\otimes \id)\overbar\Delta)\;\cap\;\ker((\id\otimes\overbar\Delta^{(r)})\overbar\Delta)
\;=\;\ker(\overbar\Delta^{(r+1)}).
\end{align*}

\end{proof}

\subsection{Comodules and coderivations}

A bicomodule over a coalgebra $C$ is a graded vector space $D$ with a coaction $\gamma:D\to C\otimes D\otimes C$, that is, a linear map satisfying
\begin{equation*}
  (\varepsilon\otimes\id\otimes\varepsilon)\gamma=\id,\qquad\qquad
 (\id\otimes\gamma\otimes\id)\gamma= (\Delta\otimes\id\otimes\Delta)\gamma.
\end{equation*}
Observe that the coaction $\gamma$ determines left and right coactions
\[(\id\otimes\id\otimes\varepsilon)\gamma\colon D\to C\otimes D,\qquad
  (\varepsilon\otimes\id\otimes\id)\gamma\colon D\to D\otimes C.\]
For example, any coalgebra can be regarded as a bicomodule over itself, with left and right coactions both given by the comultiplication, and $\gamma=\Delta^{(2)}$.

A coderivation on a $C$-bicomodule $D$
is a linear map  $d\colon D\to C$, not necessarily of degree 0, satisfying
\begin{align*}
\Delta d&\;\;=\;\;(\varepsilon\otimes d\otimes\id+\id\otimes d\otimes\varepsilon)\gamma
\intertext{where we have omitted the isomorphisms \eqref{Ktrivial}. This implies $\varepsilon d=0$ and indeed}
          \Delta^{(k-1)}d&\;\;=\;\,\sum_{\myclap{\substack{i,j\geq 0\\i+1+j=k}}}\;
(\Delta^{(i-1)}\otimes d\otimes \Delta^{(j-1)})\,\gamma,\qquad(k\geq 0).
\end{align*}

\subsection{Cofree coalgebras}\label{sub:tensorcoalg}
  The \emph{tensor coalgebra} on a graded vector space $V$ is given by
\[
T^c(V)=\bigoplus_{k=0}^\infty
                V^{\otimes k}
\]
with counit defined by the $0$-th projection $\varepsilon=p_0\colon T^c(V)\to V^{{\otimes 0}}=\K$, and comultiplication defined by deconcatenation
\[
  \Delta(v_1\cdots v_k)\;=\;\sum_{i=0}^kv_1\cdots v_i\;\otimes \;v_{i+1}\cdots v_k.
\]
As usual when working in $T^c(V)\otimes T^c(V)$  we write $v_{1}\cdots v_{k}\in T^{c}(V)$ for $v_{1}\otimes\dots\otimes v_{k}\in V^{\otimes k}$, and write $1=1_{\K}$ for the unit, given by the empty product, in the case $k=0$.
%The grading on $T^c(V)$ is given by $\displaystyle |v_1\dots v_k|=\sum_{i=1}^k|v_i|$.
The reduced tensor coalgebra is given by
\[
  \overbar T^c(V)=\bigoplus_{k=1}^{\infty}V^{{\otimes k}},
\qquad\qquad
  \overbar \Delta(v_1\cdots v_k)=\sum_{i=1}^{k-1}v_1\cdots v_i\;\otimes \;v_{i+1}\cdots v_k.
\]
The unital coalgebra $T^c(V)$ is conilpotent, its space of primitives is $V^{\otimes 1}=V$, and it has filtration
\[
F_rT^c(V)\;=\;\bigoplus_{k=0}^rV^{\otimes k}.
\]
\begin{remark}[Cofreeness properties]\label{rk:tensorcoalg}
  Together with the first projection $p_1\colon T^c(V)\to V$ onto the primitives, the tensor coalgebra defines the free coalgebra in the category of conilpotent coalgebras:
\begin{enumerate}
  \item[a)] One can write the $k$-th projection as $p_k=p_1^{\otimes k}\Delta^{(k-1)}\colon T^c(V)\to V^{\otimes k}$ and
        the identity map as
        \[
        \qquad
    \id\,=\sum_{k=0}^{\infty}p_1^{\otimes k}\Delta^{(k-1)}\colon T^c(V)\to T^c(V).
        \]
  \item[b)] Any unital coalgebra map $f\colon C\to T^c(V)$ is determined by $f_1=p_1 f\colon C\to V$ via
\[
f \;= \;\sum_{k=0}^\infty p_1^{{\otimes k}} \Delta^{(k-1)}  f
   \;= \;\sum_{k=0}^\infty f_1^{\otimes k}\Delta_C^{(k-1)}\colon C\to  T^c(V).
\qquad\qquad
\]
        Conversely if $C$ is a conilpotent coalgebra then the formula on the right hand side gives a well-defined unital coalgebra map  $f=\sum f_k\colon C\to T^c(V)$ extending any linear map $f_1\colon C\to V$ with $f_{1}(1)=0$. Hence projection to the primitives gives a bijection
        \[
        p_1^*\colon\hom_{\textrm{conil}}(C,T^cV)\;\cong\; \hom_\K(\overbar C,V).
        \]
  \item[c)]
        Any coderivation $d$ on a $T^c(V)$-bicomodule $D$ is determined by $d_1=p_1d\colon D\to V$, via
\begin{align*}
d\;=\; 
\sum_{k=0}^\infty
  p_1^{\otimes k} \Delta^{(k-1)} d        &\;=\;\sum_{k=0}^\infty\;p_1^{\otimes k}\;\sum_{\myclap{\substack{i,j\geq 0\\i+1+j=k}}}\;
(\Delta^{(i-1)}\otimes d\otimes \Delta^{(j-1)})\,\gamma
  \;=\;
       \sum_{\myclap{\substack{i,j\geq 0
  }}}\;
(p_i\otimes d_1\otimes p_j)\,\gamma
\end{align*}
Conversely the formula on the right hand side gives a well-defined coderivation $d\colon D\to T^c(V)$ extending any linear map $d_1\colon D\to V$. Hence projection to the primitives gives a bijection
        \[
        p_1^*\colon\textrm{coder}(D,T^cV)\;\cong\; \hom_\K(D,V).
        \]
\end{enumerate}
\end{remark}
A unital coalgebra is termed cofree (more precisely: cofree among conilpotent coalgebras) if it is isomorphic to a tensor coalgebra $T^c(V)$, or equivalently if it comes equipped with a projection $p_1$ to its space of primitives having the same cofreeness properties as above.

\subsection{Bialgebras and unital infinitesimal bialgebras}
Recall that an \emph{algebra} (or more explicitly, a unital associative algebra over $\K$) is a graded vector space $B$ equipped with a multiplication $\mu\colon B\otimes B\to B$
and a unit $\eta\colon \K\to B$,  $\eta(1_{\K})=1$,
satisfying associativity and unit laws.
A differential graded algebra is an algebra $(B,\mu,1)$ equipped with a map $d\colon B\to B$ of degree $-1$ satisfying $d^{2}=0$ and
$d\mu=\mu(\id\otimes d+d\otimes\id)$.

An \emph{algebra map} is a linear map preserving the unit and multiplication.

An algebra $B$ is \emph{augmented} or \emph{counital} if it is equipped with an algebra map $\varepsilon\colon B\to\K$.

The tensor product $A\otimes B$ of algebras has a canonical algebra structure
\[
  A\otimes B\:\otimes \:A\otimes B
  \xrightarrow{(23)}
  A\otimes A\:\otimes \:B\otimes B
\xrightarrow{\mu_{A}\otimes\mu_{B}}A\otimes B.
\]

\begin{definition}\label{def:bialg}
  Let $B$ be a graded vector space with structures of
  an algebra  $(B,\mu,\eta)$
  and of
  a coalgebra $(B,\Delta,\varepsilon)$.
  Then the
  structure $(B,\mu,\eta,\Delta,\varepsilon)$ is termed
  \begin{enumerate}
    \item \label{BIalg}
          a  \emph{bialgebra}
          if $\mu\colon B\otimes B\to B$ and $\eta\colon\K\to B$ are coalgebra maps,
          or equivalently
          if $\Delta\colon B\to B\otimes B$ and $\varepsilon\colon B\to\K$ are algebra maps,
    \item \label{UIalg} a \emph{unital infinitesimal bialgebra} if
          \begin{align*}
            \Delta(xy)\;&=\;
 x_{(1)}\otimes x_{(2)}y
 \;+\;
 xy_{(1)}\otimes y_{(2)}
      \;-\;x\otimes y.
          \end{align*}
          An equivalent definition is given in Lemma \ref{lem:inf-alt} and Corollary \ref{cor:inf-alt}.
\end{enumerate}
    We use Sweedler notation for comultiplication and denote the multiplication $\mu$ by juxtaposition.

    A bialgebra or a unital infinitesimal bialgebra is termed \emph{conilpotent} if it is conilpotent as a unital coalgebra, and is termed \emph{cofree} if is is cofree as a unital coalgebra.

  \end{definition}
Any conilpotent bialgebra has a canonical Hopf algebra structure, though not all Hopf algebras are conilpotent.

\begin{example}[Shuffle bialgebra]\label{ex:sha}
The shuffle bialgebra  $T^{\sha}(V)$ on a graded vector space $V$ has underlying coalgebra $T^c(V)$ with $\Delta$ given by deconcatenation as described in section \ref{sub:tensorcoalg}. The shuffle product $\sha\colon T^c(V)\otimes T^c(V)\to T^c(V)$
is unique coalgebra map that extends, via the cofreeness property of Remark \ref{rk:tensorcoalg} (b), the linear map
\begin{equation}
\label{eq:sha1}
\sha_1=
\varepsilon\otimes p_1+p_1 \otimes\varepsilon
\colon \overbar{T^c(V)\otimes T^c(V)}\longrightarrow \K\otimes V\;\oplus\; V\otimes \K\longrightarrow V.
\end{equation}
Observe that $\sha_1$ is zero on $\overbar T^c(V)\otimes \overbar T^c(V)$.
The unit laws and associative law for $\sha$ follow by standard cofreeness arguments:
the coalgebra maps $\sha(\eta\otimes\id)$%\lambda^{-1} Ktrivial
and $\sha(\id\otimes\eta)$ are the identity since their compositions with $p_{1}$ are just $p_{1}$,
and similarly the coalgebra maps $\sha(\sha\otimes\id)$, $\sha(\id\otimes\sha)\colon T^c(V)^{\otimes 3}\to T^c(V)$ are equal since one can check their projections to $V$ are equal.
Observe that $x\sha_1y$ is symmetric in $x,y$ and is non-zero only when $x y$ is a tensor of length exactly 1. It follows that the extension $x \sha y$ is graded commutative and that each projection $\sha_{k}=\sha_1^{\otimes k}\Delta^{(k-1)}$ is non-zero only when $x y$ is a tensor of total length exactly $k$. On elements we have
\begin{equation}
\label{eq:sha}
v_1\cdots v_i\;\sha\; v_{i+1}\cdots v_k\;=\;\sum_{\sigma}(-1)^\kappa%{\sgn(v_1,\dots,v_k;\sigma)}
\; v_{\sigma^{-1}(1)}\cdots v_{\sigma^{-1}(k)}
  .
  \end{equation}
  Here the sum is over all $(i,k-i)$-shuffles, that is, permutations $\sigma\in\Sigma_k$ satisfying $\sigma(j)<\sigma(j+1)$ for all $j\neq i$, $1\leq j\leq k-1$, with the Koszul signs that were defined in  \eqref{eq:koszul-sgn}.
\end{example}

The notion of unital infinitesimal bialgebra was introduced in \cite{Lod,LodayRonco2006}, modifying the non-unital definition of infinitesimal bialgebras due to Joni and Rota in \cite{JoniRota}. The unital infinitesimal relation can be expressed
  \begin{align}  \label{eq:inf-deltabar}
    \overbar\Delta(xy)&=
    \overbar\Delta(x)(1\otimes y)
    +
    (x\otimes 1)\overbar\Delta(y)
    +
    x\otimes y
\intertext{and more generally one has}
  \overbar\Delta^{(n)}(xy)&=\sum_{\myclap{\substack{r,s\geq0\\r+s=n}}}
    \bigl(\overbar\Delta^{(r)}(x)\otimes
    %\overbrace{1\otimes \dots\otimes 1}^s\,
1^{\otimes s}
    \bigr)
    \bigl(
    1^{\otimes r}
    %\overbrace{1\otimes \dots\otimes 1}^r
    \otimes\overbar\Delta^{(s)}(y)\bigr)  +\sum_{\myclap{\substack{r,s\geq0\\r+s=n-1}}}\overbar\Delta^{(r)}(x)\otimes\overbar\Delta^{(s)}(y)
    .
    \label{eq:inf-deltabar-n}
  \end{align}
  \begin{lemma}\label{lem:inf-delta-prim}
  Let $(W,\qdot,1,\Delta,\varepsilon)$ be a unital infinitesimal bialgebra.
\begin{enumerate}
\item It $v_{1},\dots,v_{k}\in W$ are primitive then
\[
\Delta(v_{1}\qdot\cdots\qdot v_{k})\;\;=\;\;
\sum_{i=0}^kv_1\qdot\cdots\qdot v_i\;\otimes\; v_{i+1}\qdot\cdots\qdot v_k.
        \]
  \item If $x,y\in \overbar W$
        with $\overbar\Delta^{(i)}(x)=0$ and $\overbar\Delta^{(j)}(y)=0$  then
        $\overbar\Delta^{(i+j)}(x\qdot y)=0$.
\newline
        In particular the conilpotent radical $R(W)$ is a unital infinitesimal subbialgebra of $W$.
\end{enumerate}
\end{lemma}
\begin{proof}
  (1): For $x=v_{1}\qdot\cdots\qdot v_{k-1}$, and $y=v_{k}$ primitive, we have
$
     \overbar\Delta(x\qdot v_{k})
     =\overbar\Delta(x)\qdot(1\otimes v_{k})+x\otimes v_{k}$ from   \eqref{eq:inf-deltabar},
     and the result follows
   by induction on $k$.\quad
(2): This is clear from   \eqref{eq:inf-deltabar-n}.
 \end{proof}
In the following Lemma and Corollary we see that a unital infinitesimal bialgebra structure can also be expressed in terms of the comultiplication $\Delta\colon B\to B\otimes B$ being an algebra map
  or the multiplication $\mu\colon B\otimes B\to B$ being a coalgebra map if we do not use the canonical (co)multiplication on the vector space $B\otimes B$
 but an alternative structure:
\begin{lemma}\label{lem:inf-alt}
  Let $(B,\Delta,\varepsilon,1)$ be a unital coalgebra. Then
  \begin{align*}
    \Delta'(x\otimes y)&=
                         (x_{(1)}\otimes 1)\otimes (x_{(2)}\otimes y)
+
                         (x\otimes y_{(1)})\otimes (1\otimes y_{(2)})
-
                         (x\otimes 1)\otimes (1\otimes y),
  \end{align*}
  defines a unital coalgebra structure
  $(B^{\otimes 2},\Delta',\varepsilon^{(2)},1\otimes 1)$ with the canonical unit and counit.\smallskip

  Now let $(B,\qdot,1,\varepsilon)$ be a counital algebra. Then
there is a counital algebra structure defined on $B^{\otimes 2}$ by $1^{\otimes 2}$, $\varepsilon^{\otimes 2}$ and
  $\qdot'\colon B^{\otimes 2}\otimes B^{\otimes 2}\to B^{\otimes 2}$ where
  \begin{align*}
    (x_1\otimes x_2)\qdot' (y_1\otimes y_2)
&=
\varepsilon(y_1)\;x_1\otimes x_2{\qdot} y_2
\,+\,
\varepsilon(x_2)\;x_1{\qdot} y_1\otimes y_2
\,-\,
\varepsilon(x_2 {\qdot} y_1)\;x_1\otimes y_2.
  \end{align*}
\end{lemma}
\begin{proof}
For the first part, we check that $(\Delta'\otimes\id)$ and  $(\id\otimes\Delta')$ applied to $\Delta'(x\otimes y)$ both give
  \begin{align*}
     (x_{(1)} \otimes 1)\otimes (x_{(2)} \otimes 1)\otimes (x_{(3)} \otimes y)
    &+ (x \otimes y_{(1)})\otimes (1 \otimes y_{(2)})\otimes (1 \otimes y_{(3)})
  \\{}+   (x_{(1)} \otimes 1)\otimes (x_{(2)} &\otimes y_{(1)})\otimes (1 \otimes y_{(2)})
    \\
{}-  (x \otimes 1)\otimes ( 1\otimes y_{(1)})\otimes (1 \otimes y_{(2)})
  &-  (x_{(1)} \otimes 1)\otimes (x_{(2)} \otimes 1)\otimes (1 \otimes y)
  \end{align*}
so $\Delta'$ is coassociative. It is easy to verify that $(\varepsilon^{(2)}\otimes\id)$ and  $(\id\otimes\varepsilon^{(2)})$ applied to $\Delta'(x\otimes y)$ both give $x\otimes y$ and that $1\otimes 1$ is grouplike.

The second part is similar.
\end{proof}
\begin{corollary}\label{cor:inf-alt}
  Suppose $B$ is a graded vector space equipped with a
   unital counital algebra structure $(B,\qdot,1,\varepsilon)$
   and a unital counital coalgebra structure $(B,\Delta,\varepsilon,1)$.
  Then the following are equivalent:
  \begin{enumerate}
    \item $(B,\qdot,1,\Delta,\varepsilon)$ is a unital infinitesimal bialgebra.
    \item The multiplication is a coalgebra map $\qdot\colon(B\otimes B,\Delta')\to (B,\Delta)$.
    \item The comultiplication is an algebra map $\Delta\colon(B,\qdot)\to (B\otimes B,\qdot')$.
  \end{enumerate}
\end{corollary}

 \begin{example}[Fundamental infinitesimal bialgebra]\label{ex:fund-inf-bialg}
   Let $V$ be a graded vector space and consider the unital infinitesimal bialgebra $(T^{fc}(V),\qdot,1,\Delta,\varepsilon)$ defined as follows. As a graded algebra it is the \emph{tensor algebra} $T(V)=\bigoplus_{n\geq0}V^{\otimes n}$ with multiplication $\qdot$ defined by concatenation of tensors. This is the free algebra on $V$ so by Corollary \ref{cor:inf-alt} (3) a unital infinitesimal comultiplication $\Delta$ can be specified by giving its values on the generators. In particular we can define
   \[
\Delta\colon(T(V),\qdot)\to (T(V)\otimes T(V),\qdot'),\qquad \Delta(v)=v\otimes 1+1\otimes v\quad (v\in V).
\]
so that the generators are primitive. Now by Lemma \ref{lem:inf-delta-prim} (1) we see that $\Delta$ coincides with the comultiplication given by deconcatenation of tensors.
That is, the fundamental infinitesimal bialgebra is at once the free
algebra and the cofree conilpotent coalgebra on $V$.
\end{example}
\begin{prop} The functors $T^{fc}$ and $\P$ are adjoint:
  for any vector space $A$ and any unital infinitesimal bialgebra $W$, there is a natural bijection between linear maps $f\colon A\to\P(W)$ and unital infinitesimal bialgebra homomorphisms $F\colon T^{fc}(A)\to W$,
  such that $f(v)=F(v)$ for all $v\in A$.
The unit of the adjunction is the natural isomorphism $A\cong \P (T^{fc}(A))$.
\end{prop}
\begin{proof}
  Define $F$ as the algebra homomorphism extending $A\stackrel f\to \P(W)\stackrel\subseteq\to W$ to the free algebra,
  \[
F(v_1\qdot\cdots\qdot v_{k})=fv_1\qdot\cdots\qdot fv_{k}.
  \]
 This is a bialgebra homomorphism by Lemma \ref{lem:inf-delta-prim} (1). Conversely, restricting $F$ to $A$ defines the linear map $f$, since $F$ sends primitives to primitives.
\end{proof}

For the counit, consider $f_{W}=\id\colon \P(W)\stackrel=\to\P(W)$ and its extension
\begin{equation}\label{F}
F_{W}\colon  T^{fc}(\P(W))
  \to W.
\end{equation}
\begin{prop}\label{prop:LR-inf-class}
The natural homomorphisms $F_W$ are injective and their image is the conilpotent radical $R(W)$. In particular
  any conilpotent unital infinitesimal bialgebra is naturally isomorphic to the fundamental infinitesimal bialgebra on its primitives.
\end{prop}
\begin{proof}
  Clearly $F_{W}$ has conilpotent image and so factors
  \[T^{fc}(\P(W))=T^{fc}(\P(R(W)))
    \to R(W)\subseteq W.\]
  Now observe that
  $F\colon T^{fc}(\P(R(W)))
    \to R(W)$ is an isomorphism, by \cite[Theorem 2.6]{LodayRonco2006}.
  \end{proof}
  In other words:
\begin{corollary}\label{cor:corefl}
The functor $T^{fc}$ embeds the category of vector spaces as
the full coreflective subcategory of unital infinitesimal bialgebras whose objects are conilpotent
or, equivalently, cofree.
\end{corollary}

\section{\ainfty algebras, multibrace algebras and \binfty algebras}

\subsection{\ainfty algebras}
The notion of $A_\infty$ algebra is encoded in differential graded structures on the tensor coalgebra.

Consider the tensor coalgebra $T^c(V)$ on a graded vector space $V$ as a bicomodule over itself.

Any coderivation $d\colon T^c(V)\to T^c(V)$, and each the projections $d_k=p_kd\colon T^c(V)\to V^{\otimes k}$,
are determined by the projection to the primitives $d_1=p_1d\colon T^c(V)\to V$.  Explicitly, if we denote
\[d_1=\left(m_n\colon V^{\otimes n}\to V\right)_{n\geq 0},\qquad\qquad
d_k=\left(m^{k}_n\colon V^{\otimes n}\to V^{\otimes k}\right)_{n\geq0},
\]
then by the cofreeness property of Remark~\ref{rk:tensorcoalg} (c),
\begin{equation}\label{eq:mnk}
m^k_n=\sum_{\myclap{\substack{i,j\geq0\\i+1+j=k}}}\id_V^{\otimes i}\otimes m_{n-i-j}\otimes \id_V^{\otimes j}
\colon V^{\otimes n}\to V^{\otimes k}.
\end{equation}
A coderivation $d$ on $T^c(V)$ is termed a \emph{differential} if it has degree $-1$ and satisfies $d(1)=0$ and $d^{2}=0$. The condition $d(1)=0$ says $m_0=0$ and $m_n^k=0$ for $k>n$; we may regard a differential as a map
\begin{gather}
  d\colon\overbar T^{c}(V)\to \overbar T^{c}(V)\intertext{ determined by }
  d_1=p_{1}d=\sum_{n\geq 1}m_n\colon \overbar T^{c}(V)\to V.
\end{gather}
The condition $d^2=0$ therefore says
\begin{equation}\label{eq:a-infty}
  \qquad\qquad\sum_{\substack{i,j\ge0\\i+j<n}}m_{i+1+j}\left(\id_{V^{\otimes i}}\otimes m_{n-i-j}\otimes \id_{V^{\otimes j}}\right)=0\colon V^{\otimes n}\to V\qquad (n\geq 1).
\end{equation}
As the maps $m_{n}$ have degree $-1$ we note that the usual
Koszul signs appear in the summations of \eqref{eq:mnk}  or \eqref{eq:a-infty} when they are applied to elements.

\begin{definition}[$A_\infty$ algebra]
  An $A_\infty$ algebra is a graded vector space $V$ together with a differential  $d$ on $T^c(V)$, or equivalently a graded vector space $V$ with a sequence  $d_1=(m_1,m_2,m_3,\dots)$ of degree $-1$ maps $m_n\colon V^{\otimes n}\to V$ satisfying the relations \eqref{eq:a-infty}.
\end{definition}
\begin{remark}\label{rk:shiftA}
This notion might more properly be called a \emph{shifted $A_{\infty}$ algebra}: an $A_\infty$ algebra is classically defined via a differential on the bar construction rather than on the tensor coalgebra. The definitions are essentially equivalent and we adopt the latter as it simplifies degrees and signs.
Let $V[1]=\K[1]\otimes V$, where $\K[1]$ is a copy of the field $\K$ concentrated in degree 1, and let  $s\colon V\to V[1]$, $sv=1\otimes v$,
  be the canonical degree 1 suspension isomorphism.
  Then a classical $A_\infty$ structure on $V$ corresponds to a shifted $A_\infty$ structure on $V[1]$,
  and $V$ is thus equipped with a sequence of degree $n-2$ maps $s^{-1}m_{n}s^{\otimes n}\colon V^{\otimes n}\to V$
satisfying relations analogous to \eqref{eq:a-infty} but with additional Koszul signs that arise from rearranging the intervening suspension maps.
\end{remark}
\begin{example}[Trivial $A_\infty$ structure]\label{ex:trivAinf}
  If $V$ comes equipped with an endomorphism $\partial\colon V\to V$ of degree $-1$ then the map
  \[d_1=(\partial,0,0,\dots)=\partial p_1\colon T^c(V)\longrightarrow V\]
  extends to a unique coderivation $d=\sum d_k$ on $T^c(V)$ of degree $-1$. The extension satisfies $d^2=0$ if and only if $\partial^2=0$, that is, $(V,\partial)$ is a chain complex.
  The differential on $T^c(V)$ in this case is the \emph{trivial $A_\infty$ structure on $V$} and is given explicitly by the structure maps
  \[
    m^k_n=0 \;\;\; (k\neq n),\qquad
    m^k_k=\sum_{i=0}^{k-1}\id_V^{\otimes i}\otimes \partial\otimes \id_V^{\otimes k-i-1}
    \colon V^{\otimes k}\to V^{\otimes k}.
  \]

\end{example}

\subsection{Multibrace algebras}
The notion of multibrace algebra arises from bialgebra structures on the tensor coalgebra.
If $V$ is a graded vector space then a bialgebra structure on the tensor coalgebra $(T^c(V),1,\Delta,\varepsilon)$ is a %unital
coalgebra map $\mu\colon T^c(V)\otimes T^c(V)\to T^c(V)$ of degree zero satisfying the associative and unit laws.

As the tensor coalgebra is cofree, the coalgebra map $\mu$ and each of its projections
$\mu_r=p_r\mu\colon T^c(V)\otimes T^c(V)\to V^{\otimes r}$ are determined by $\mu_{1}$. Let us denote
their components by\[
\mu_{1}=\left(m_{i,j}    \colon V^{\otimes i}\otimes V^{\otimes j}\to V   \right)_{i,j\ge0}
,\qquad
\mu_{r}=\left(m_{i,j}^{r}\colon V^{\otimes i}\otimes V^{\otimes j}\to V^{\otimes r}    \right)_{i,j\ge0}
.\]

The unit law for the multiplication $\mu$ says that it is determined by its restriction to the reduced coalgebra, see Remark \ref{rk:unit-law}, and hence by the components $m_{i,j}$ with $i,j\geq 1$,
\begin{equation}
  \begin{tikzcd}
    \overbar T^{c}(V)\otimes \overbar T^{c}(V)\ar[r]
    \ar[rd,"{\sum\limits_{\!\!\!i,j\geq 1\!\!\!} m_{i,j}\!\!\!\!\!}"']
    &\ar[d,"p_{1}"] T^{c}(V)
\\&V,\end{tikzcd}
\end{equation}
as the components with $i$ or $j$ zero are just projections to $V$ of the identifications $\K\otimes T^c(V)= T^c(V)$ and  $T^c(V)\otimes\K= T^c(V)$ respectively, cf. \eqref{Ktrivial},
  \begin{equation}\label{m0n}
    \left\{\begin{array}{c}
      m_{0,1}=%\lambda Ktrivial
      \id
      \colon\K\otimes V^{\phantom{\otimes n}}\stackrel=\longrightarrow V,
      \\
      m_{0,n}=\,0\,\colon\K\otimes V^{\otimes n}\longrightarrow V,
          \end{array}\right.
\qquad
\left\{\begin{array}{cl}
  m_{1,0}=%\rho Ktrivial
  \id
  \colon V^{\phantom {\otimes n}}\otimes \K\stackrel=\longrightarrow V.&\\
m_{n,0}=\,0\,\colon V^{\otimes n}\otimes\K\longrightarrow V,&
                     \qquad(n\neq1).
   \end{array}\right.
\end{equation}

The associativity law for the multiplication $\mu$ is equivalent to $\mu_1(\mu\otimes\id)=\mu_1(\id\otimes\mu)$. To make this explicit we introduce some auxiliary notation.
\begin{notation}
  Denote by $\compositions ir\subset \mathbb N^{r}$ the set of sequences $\underline i=(i_1,\dots,i_r)$ of $r$ nonnegative integers whose sum is $i$.
  Now each iterated comultiplication $\Delta^{(r-1)}$ %on $T^c(V)$ and
  on $T^c(V)\otimes T^c(V)$ has as components
\begin{gather*}
\Delta_{r}^{\underline i,\underline j}\colon V^{\otimes i}\!\otimes\! V^{\otimes j}=%\stackrel=\longrightarrow
(V^{\otimes i_1}\!\otimes\! \cdots\!\otimes\! V^{\otimes i_r}
)\otimes(
V^{\otimes j_1}\!\otimes\!\cdots\! \otimes\! V^{\otimes j_r})
\stackrel\sigma\longrightarrow
(V^{\otimes i_1}\!\otimes\! V^{\otimes j_1})\otimes\cdots\otimes (V^{\otimes i_r}\!\otimes \!V^{\otimes j_r})
\end{gather*}
for all $i,j\geq0$ and
%$\underline i\in \compositions ir$ and
$(\underline i,\underline j) \in \twocompositions ijr$.
When $\Delta_{r}^{\underline i,\underline j}$ is applied to elements the Koszul signs   \eqref{eq:koszul-sgn} appear.
\end{notation}

By Remark \ref{rk:tensorcoalg} (b) we have $\mu_r=p_r\mu=\mu_1^{\otimes r}\Delta^{(r-1)}\colon T^c(V)\otimes T^c(V)\to V^{\otimes r}$, and so its components can be written as
  \begin{equation}\label{eq:mu-r}
    m_{i,j}^{r}\;\;=\;\;\sum_{(\underline i,\underline j)\in   \twocompositions ijr}
    (m_{i_1,j_1}\otimes
    \cdots\otimes
    m_{i_r,j_r})\Delta_{r}^{\underline i,\underline j}\colon V^{\otimes i}\otimes V^{\otimes j}\to V^{\otimes r}
    .  \end{equation}
The components $m_{i,j}^{i+j}$ coincide with the shuffle product \eqref{eq:sha},  and  $m_{i,j}^r=0$ if $r> i+j$.

We can now express the associativity of $\mu$ in terms of the $m_{i,j}$: for each $i,j,k\ge 1$ there is an equality
\begin{align}\label{eq:Rijk}
    \sum_{\myclap{\substack{r\geq1\\ (\underline i, \underline j)\in   \twocompositions ijr}}}
    m_{r,k}&\left((m_{i_1,j_1}\otimes
    \cdots\otimes
    m_{i_r,j_r})\Delta_{r}^{\underline i,\underline j}\,\otimes\,\id_V^{\otimes k}\right)
  \\[-1em]&\quad=\;
      \sum_{\myclap{\substack{s\geq1\\ (\underline j,\underline k)\in   \twocompositions jks}}}
    m_{i,s}\left(\id_V^{\otimes i}\,\otimes\,(m_{j_1,k_1}\otimes
    \cdots\otimes
    m_{j_s,k_s})\Delta_{s}^{\underline j,\underline k}
  \right)\nonumber
\end{align}
of linear maps $    V^{\otimes i}\otimes V^{\otimes j}\otimes V^{\otimes k}\to V$, where terms $m_{n,0}$ and $m_{0,n}$ are given by \eqref{m0n}.

\begin{definition}[Multibrace algebra]
  A multibrace algebra is a graded vector space $V$ with a coalgebra map $\mu\colon T^c(V)\otimes T^c(V)\to T^c(V)$ satisfying the associative and unit laws.

  Equivalently, it is a graded vector space $V$ endowed with a family of degree zero multilinear maps $m_{i,j}\colon V^{\otimes i}\otimes V^{\otimes j}\to V$, $i,j\geq 1$ %0
%  such that $m_{0,n}$ and  $m_{0,n}$ are given by \eqref{m0n}, and
  such that the relations \eqref{eq:Rijk} hold for $i,j,k\ge1$.
\end{definition}
%\begin{remark}
The definition of multibrace algebra in the ungraded world was given in \cite[Definition 1.5]{LodayRonco2006}, where it was termed (non-differential) $\mathbf B_{\infty}$ algebra.
As Koszul signs do not appear in the ungraded context the symmetry isomorphisms $\sigma$ and the maps $\Delta_{r}^{\underline i,\underline j}$ involved in the relations \eqref{eq:Rijk}, there termed $R_{ijk}$, were left implicit.
%\end{remark}
\begin{example}[Trivial multibrace algebra]\label{ex:trivMB}
  Consider a multibrace structure $\mu$ on $V$ in which all components $m_{i,j}$ of $\mu_{1}$  (except $m_{0,1}$ and $m_{1,0}$) are zero. Then $\mu_{1}$ is the linear map $\sha_1$ considered in Example \ref{ex:sha} \eqref{eq:sha1},
  and \eqref{eq:sha} gives the multiplication $\mu=\sha$.
    That is, we can identify the trivial multibrace algebra $V$ with the shuffle bialgebra $T^{\sha}(V)$.
\end{example}

\begin{example}[Quasi-shuffles]\label{ex:stuf}
  If $V$ carries a binary operation $\pdot\colon V\otimes V\to V$
  then we can upgrade
  the shuffle bialgebra and replace \eqref{eq:sha1} by
\begin{equation}
  \osha_1
  ={\sha_1}+{p_{1}\pdot p_{1}}
    \colon  \overbar{T^c(V)\otimes  T^c(V)}
  \to
  \K\otimes V\;\oplus\; V\otimes \K \;\oplus\; V\otimes V
  \to
  V
\end{equation}
This determines a unique coalgebra map
$\osha\colon T^c(V)\otimes T^c(V)\to T^c(V)$
which satisfies the unit law, and which is associative if and only if the binary operation $\pdot$ was.
Thus we have a multibrace algebra in which $m_{1,1}=\pdot$ and all other operations $m_{i,j}$ except $m_{0,1}$ and $m_{1,0}$ are zero. If the multiplication $\pdot$ is zero we recover the previous example.

If the operation $\pdot$ is associative and graded commutative then so is $\osha$ and this multibrace algebra is a
\emph{quasi-shuffle} (or \emph{stuffle}) algebra, compare \cite{loday-stuf}\footnote{One should interpret \cite[Proposition 1.3]{loday-stuf} and its proof with care; the diagram in the proof
  does not commute on elements of the form $a\otimes 1_\K$ for example.}.

Suppose the multiplication $\pdot$ is unital, with unit $\eta(1_\K)=1\in V$.
If $\widetilde{p_{1}}$ is the unital linear map
\begin{equation}\label{eq:unital p1}
  \widetilde{p_{1}}
  =\eta\varepsilon+p_{1}\colon T^c(V)
  \to V
\end{equation}
then $\osha_{1}$ is just the restriction to
$\overbar{T^c(V)\otimes  T^c(V)}$ of
\begin{equation}
\widetilde{p_{1}}\pdot\widetilde{p_{1}}  \colon
T^c(V)\otimes  T^c(V)\to V\otimes V\to V.
\end{equation}
\end{example}
\begin{example}\label{ex:lr-mb}
  If $V$ is endowed with two
  associative binary operations $\pdot$ and $\qdot$,
unrelated except that they share a common unit,
   it turns out that $V$ has a unique multibrace structure $\mu_{\pdot}^{\qdot}$  whose components $m^{r}_{i,j}\colon V^{\otimes i}\otimes V^{\otimes j}\to V^{\otimes r}$ satisfy
  \begin{equation}\label{eq:mu-r-lift}
    \sum_{r\ge1}\qdot^{(r-1)}m^r_{i,j}
    \;=\;
    \pdot(\qdot^{(i-1)}\otimes\qdot^{(j-1)})
    \colon V^{\otimes i}\otimes  V^{\otimes j}\longrightarrow  V.
  \end{equation}
  It is sufficient to specify the components $m_{i,j}=m^{1}_{i,j}$ which must satisfy $m_{0,0}=0$ and,
  by \eqref{eq:mu-r},
\begin{equation}\label{eq:mij}
    m_{i,j}\;\;
=\;\;    \pdot(\qdot^{(i-1)}\otimes\qdot^{(j-1)})\;\;-\;\;\sum_{r\ge2}\qdot^{(r-1)}\sum_{ (\underline i, \underline j)\in   \twocompositions ijr}
    (m_{i_1,j_1}\otimes
 \cdots   \otimes
    m_{i_r,j_r})\Delta_{r}^{\underline i,\underline j}
\end{equation}
This recursive formula to define maps $m_{i,j}$ %and hence $\mu_{\pdot}^{\qdot}$
was studied
in the ungraded situation by Loday and Ronco
\cite[Section 3.3 and Proposition 3.4]{LodayRonco2006}.
In Theorem \ref{thm:lr-mb} we will give a simple argument to show that the operations $(m_{i,j})$ so defined do indeed satisfy the multibrace axioms \eqref{eq:Rijk}.
\end{example}

\subsection{\binfty algebras}
The notion of $B_\infty$ algebra is encoded in differential graded bialgebra structures on the tensor coalgebra.
If $V$ is a graded vector space then a differential graded bialgebra structure on $(T^c(V),1,\Delta,\varepsilon)$ is a differential $d\colon T^c(V)\to T^c(V)$ together with a coalgebra map $\mu\colon T^c(V)\otimes T^c(V)\to T^c(V)$ with such that that $(T^c(V),\mu,d,1)$ is a differential graded algebra. That is, a $B_\infty$ structure is an $A_\infty$ structure $d$ together with a compatible multibrace structure $\mu$.

If $d$ and $\mu$ are given by families of multilinear maps $(m_n)$ and $(m_{i,j})$ as above then the compatibility property is equivalent to
$d_1\mu=\mu_1(d\otimes\id+\id\otimes d)$, that is, for each $i,j\geq1$ there is an equality
\begin{align}\label{eq:Dij}
  \sum_{
     \myclap{
        \substack{r\geq1\\(\underline i,\underline j)\in\twocompositions ijr}
     }
  }
  m_r(m_{i_1,j_1}\otimes
    \cdots\otimes
    m_{i_r,j_r})\Delta_{r}^{\underline i,\underline j}
    \;&=\;
  \sum_{\myclap{\substack{p,q\ge0\\p+q<i}}}
  m_{p+1+q,j}\left((\id_V^{\otimes p}\otimes m_{i-p-q}\otimes\id_V^{\otimes q})\otimes\id_V^{\otimes j}\right)
  \\[-1ex]&+\;    \sum_{\myclap{\substack{p,q\ge0\\p+q<j}}}
    m_{i,p+1+q}\left(\id_V^{\otimes i}\otimes(\id_V^{\otimes p}\otimes m_{j-p-q}\otimes\id_V^{\otimes q})\right)
  \nonumber
      \end{align}
between linear maps $    V^{\otimes i}\otimes V^{\otimes j}
\to V$.
If one evaluates these, to write the relations as equalities between elements, then Koszul signs appear.
\begin{definition}
  A $B_\infty$ structure on a graded vector space $V$ is given by
a coalgebra map $\mu\colon T^c(V)\otimes T^c(V)\to T^c(V)$ satisfying the associative and unit laws,
together with a differential $d$ on the coalgebra $T^c(V)$ which is a derivation with respect to $\mu$. Alternatively, it is given by collections of multilinear maps
 $\left(m_{i,j}\colon V^{\otimes i}\otimes V^{\otimes j}\to V\right)_{i,j\geq1}$ of degree $0$ and
 $\left(m_n\colon V^{\otimes n}\to V\right)_{n\geq1}$  of degree $-1$ that  satisfy the families of relations
 \eqref{eq:a-infty}, % \eqref{m0n},
 \eqref{eq:Rijk} and \eqref{eq:Dij}.
\end{definition}
\begin{remark}\label{rk:shiftB}
This notion might more properly be called a \emph{shifted} $B_\infty$ algebra.  The notion of $B_\infty$ algebra is usually defined via a differential bialgebra structure on the bar construction of a graded space, rather than on the tensor coalgebra as we have done here. As mentioned in Remark \ref{rk:shiftA} above, this leads to differences in degrees and signs: a $B_\infty$ structure on $V$ corresponds to a shifted $B_\infty$ structure on $V[1]$ and is therefore equipped with degree $n-2$ operations $s^{-1}m_ns^{\otimes n}\colon V^{\otimes n}\to V$
  and degree $i+j-1$ operations $s^{-1}m_{i,j}(s^{\otimes i}\otimes s^{\otimes j})\colon V^{\otimes i}\otimes V^{\otimes j}\to V$,
  satisfying relations analogous to \eqref{eq:a-infty},  %\eqref{m0n},
  \eqref{eq:Rijk} and \eqref{eq:Dij}, but with some additional Koszul signs arising from commuting the operations with the suspension maps $s$.
\end{remark}

The non-commutative quasi-shuffle algebra of Example \ref{ex:stuf}  can be combined with
the trivial $A_\infty$ algebra of Example \ref{ex:trivAinf} as follows.
\begin{prop}[Quasi-trivial $B_\infty$ algebras]\label{prop:difstuf}
  For any differential graded algebra
  $(V,\pdot,\partial)$ there is a
$B_\infty$ structure $\mu_{\pdot},d_{\partial}$ on $V$ uniquely defined by
\[
\begin{tikzcd}
  \overbar T^c(V)\otimes \overbar T^c(V)\ar[d,
"p_{1}\otimes p_{1}"'
  ]\ar[r,"\mu_{\pdot}"]&\overbar T^c(V)\ar[d,"p_{1}"]\\
  V\!\otimes\!V\ar[r,"\pdot"]
  &V,
\end{tikzcd}\qquad
\qquad
\begin{tikzcd}
  \overbar T^c(V)\ar[d,"p_{1}"']\ar[r,"d_{\partial}"]&\ar[d,"p_{1}"]\overbar T^c(V)\\
  V\ar[r,"\partial"]&V.
\end{tikzcd}
\]
All operations $m_{i,j}$ and $m_{n}$ are zero, except  $m_{1,1}=\pdot$, $m_1=\partial$, and
$m_{0,1}=m_{1,0}=\id$, see
\eqref{Ktrivial}, \eqref{m0n}.
\end{prop}
\begin{proof}
The linear maps 
  \begin{align*}
   d_1=\partial p_1\colon\, \overbar T^c(V)\,\longrightarrow &\,V,
  \\    \mu_1=p_1\pdot p_1
    \colon
    \overbar T^c(V)^{\otimes 2}
    \to&\, V,
\end{align*}
considered in Examples  \ref{ex:trivAinf}  and \ref{ex:stuf} determine  respectively
     a coderivation $d_{\partial}$ on $T^c(V)$
and a coalgebra map $\mu_{\pdot}=\osha\colon T^c(V)^{\otimes 2}\to T^c(V)$ satisfying the unit law. The  $B_\infty$ algebra axioms follow by uniqueness of extensions:
  \begin{enumerate}
    \item The binary operation $\mu_{\pdot}$ is associative if and only if $\pdot$ is associative, since $\mu_{\pdot}(\mu_{\pdot}\otimes\id)$ and $\mu_{\pdot}(\id\otimes\mu_{\pdot})$ are coalgebra maps whose projections to $V$ are
          $(p_1\pdot p_1)\pdot p_1$ and $p_1\pdot (p_1\pdot p_1)$.
     \item $d_{\partial}^{2}=0$ if and only if $\partial^2=0$, since $d_{\partial}^2$ is a coderivation satisfying $p_1d_{\partial}^2=\partial p_{1}d_{\partial}=\partial^2p_1$.
    \item $d_{\partial}$ is a derivation of $\mu_{\pdot}$ if and only if  $\partial$ is a derivation of $\pdot$, since $\mu_{\pdot}(d_{\partial}\otimes\id+\id\otimes d_{\partial})-d_{\partial}\mu_{\pdot}$ is a coderivation whose projection to $V$ is
          $\partial p_1\pdot p_1+p_1\pdot\partial p_1-\partial(p_1\pdot p_1)$.
  \end{enumerate}
  Thus given a differential graded algebra $(V,\pdot,\partial)$ we have a $B_{\infty}$ algebra $(V,\mu_\pdot,d_\partial)$.
\end{proof}

\subsection{Twistings}\label{sec:twist}

By a \emph{twisting} $\tau$ on a graded vector space $V$ we mean a coalgebra automorphism of $T^c(V)$. By cofreeness a twisting $\tau$ is determined by $\tau_1\colon T^{c}(V)\to V$, that is, by a sequence of multilinear maps $t_{n}\colon V^{\otimes n}\to V$, $n\geq1$. For simplicity we will assume the component $t_{1}=\id_{V}$. The inverse $\tau^{-1}$ is determined by $u_{n}\colon  V^{\otimes n}\to V$, $n\ge1$, that can be calculated recursively from $(\tau^{-1})_1\circ \tau=p_1$, that is, $u_{1}=\id_{V}$ and
\[
\sum_{0<r\le n} {u_r\tau_1^{\otimes r}\Delta^{(r-1)}}_{|V^{\otimes n}}\;=\;
u_n+  \sum_{\substack{0<r<n\\[0.3ex] \underline i\in   \overbar{\compositions nr}}}u_r\circ(t_{i_{1}}\otimes\cdots\otimes t_{i_{r}})\;=\;0\colon V^{\otimes n}\to V\qquad(n\ge2)
\]
where  $\overbar{\compositions nr}\subseteq \compositions nr$ is the subset of sequences of strictly positive integers.

The twistings that interest us here will arise from associative binary operations, as follows.
\begin{lemma}\label{lem:twistmult}
Suppose $V$ is a graded vector space equipped with an associative binary operation $\qdot$. Then the sequence of iterated multiplication maps $\left(\qdot^{(n-1)}\colon V^{\otimes n}\to V\right)_{n\geq 1}$ determines a twisting $\tau^{\qdot}$ whose inverse is determined by $\left((-1)^{n-1}\;\qdot^{(n-1)}\colon V^{\otimes n}\to V\right)_{n\geq1}$.

If $T^{c}(V)$ is considered as a free algebra, then the projection $\tau^{\qdot}_1=p_1\tau^{\qdot}\colon \overbar T^c(V)\to V$ is a non-unital algebra homomorphism.
If $V$ has a unit and $\widetilde{p_1}$ is as defined in \eqref{eq:unital p1}, then $\widetilde{p_1}\tau^{\qdot}\colon T^c(V)\to V$ is a unital algebra homomorphism.
\end{lemma}
\begin{proof}
  As $\qdot^{(0)}=\id_V$, we know $\tau$ is invertible.
  For the inverse see for example \cite[Section 2.2]{markl}.
  By definition the maps $p_{1}\tau^{\qdot}$ and $\widetilde{p_{1}}\tau^{\qdot}$ are given by multiplication of generators
  \[
\begin{tikzcd}
  \overbar T^{c}(V)
  \ar[r,"\tau^{\qdot}"]\ar[rd,"{(\id,\qdot,\qdot^{(2)},\dots)
  }"']
  &\overbar T^{c}(V)\ar[d,"p_{1}"]
& T^{c}(V)
  \ar[r,"\tau^{\qdot}"]\ar[rd,"{(\eta,\id,\qdot,\qdot^{(2)},\dots)}"']
  & T^{c}(V)\ar[d,"\widetilde{p_1}"]
  \\&V,&&V.
\end{tikzcd}
\]
These are just the counits of the free-forget adjunctions between vector spaces and  (non-unital or unital) algebras: the homomorphisms given by extending $\id\colon V\to V$ to the respective free algebra.
\end{proof}

The $A_{\infty}$ case of the following result was thoroughly investigated in \cite{markl}.
\begin{proposition}\label{prop:twisted}
  Given a $B_{\infty}$ algebra $(V,\mu,d)$,
  then for any twisting $\tau$ on $V$ there is a twisted
$B_{\infty}$ algebra $(V,\mu^{\tau},d^{\tau})$ defined by
\[
\begin{tikzcd}
   T^c(V)\otimes  T^c(V)\ar[d,"\tau\otimes\tau"']\ar[r,"\mu^\tau"]&  T^c(V)\ar[d,"\tau"]\\
  T^c(V)\otimes  T^c(V)\ar[r,"\mu"]&  T^c(V),
\end{tikzcd}\qquad
\qquad
\begin{tikzcd} T^c(V)\ar[d,"\tau"']\ar[r,"d^\tau"]& T^c(V)\ar[d,"\tau"]\\
 T^c(V)\ar[r,"d"]&  T^c(V).
\end{tikzcd}
\]
\end{proposition}
\begin{proof}
  The transfer of the original structure $\mu$, $d$ along the isomorphism $\tau$ gives the twisted structure $\mu^{\tau}=\tau^{-1}\mu \tau^{\otimes2}$ and $d^{\tau}=\tau^{-1}d\tau$. Since $\tau$ is a coalgebra automorphism it is clear that
  $\mu^{\tau}$ and $d^{\tau}$ still define a coalgebra map and a coderivation respectively and satisfy the  $B_{\infty}$ axioms (associativity, square zero, compatibility).
\end{proof}

\section{\binfty algebras and 2-associative differential (bi)algebras}

\subsection{The underlying \binfty algebra}\label{sec:31}
We can now return to the construction in Example \ref{ex:lr-mb} of multibrace algebras from spaces with two associative operations, and generalise it to $B_{\infty}$ algebras.

\begin{definition}[2-associative differential algebra]
  A 2-associative differential algebra is a differential graded algebra $(V,\pdot,1,\partial)$ endowed with a second associative binary operation $\qdot$, with the same unit $1$ but not required to satisfy any other compatibility relation with $\pdot$ or with $\partial$.
  Together with their homomorphisms they form a category $\dAs^{1,1}\text{-alg}$.
\end{definition}

To any 2-associative differential algebra $V$ there is an underlying $B_\infty$ structure on $V$ defined by twisting the quasi-trivial structure:

\begin{theorem}[The underlying $B_\infty$ algebra]\label{thm:lr-mb}
  For any 2-associative differential algebra
  $(V,\pdot,\qdot,1,\partial)$
  there is an underlying
   $B_\infty$ structure
  $\mu_{\pdot}^{\qdot},d_{\partial}^{\qdot}$ on $V$
  uniquely defined by
  \begin{equation}\label{eq:associated-binfty-diags}
\begin{tikzcd}
  \overbar T^c(V)\otimes \overbar T^c(V)
  \ar[d,"\tau^{\qdot}_{1}\otimes\tau^{\qdot}_{1}"']
  \ar[r,"\mu^{\qdot}_{\pdot}"]&
  \overbar T^c(V)\ar[d,"\tau^{\qdot}_{1}"]\\
  V\!\otimes\!V\ar[r,  "\pdot"]&V,
\end{tikzcd}
\qquad\qquad
\begin{tikzcd}
  \overbar T^c(V)\ar[d,"\tau^{\qdot}_{1}"']
  \ar[r,"d^{\qdot}_{\partial}"]&
  \overbar T^c(V)\ar[d,"\tau^{\qdot}_{1}"]\\
  V\ar[r,"\partial"]&V,
\end{tikzcd}
\end{equation}
where $\tau^{\qdot}$ is the twisting determined by $\qdot$.
\end{theorem}
\begin{proof}
    Twist with $\tau^{\qdot}$
  the quasi-trivial  $B_{\infty}$ structure $\mu_\pdot,d_\partial$ on $V$, by Propositions~\ref{prop:difstuf} and \ref{prop:twisted}.
\end{proof}
Homomorphisms of 2-associative differential algebras define homomorphisms of the underlying $B_\infty$ algebras, and we have a functor
\begin{equation}\label{eq:associated-b-infty-functor}
\dAs^{1,1}\text{-alg} \to B_{\infty}\text{-alg}.
\end{equation}
There are many useful reformulations of the defining property of the underlying
$B_\infty$ structure $\mu_\pdot^{\qdot},d_\partial^{\qdot}$. For example, it is the unique structure with the following equivalent properties:
\begin{itemize}
\item
$\tau^{\qdot}_1\colon (\overbar T^c(V),\mu^{\qdot}_\pdot,d^{\qdot}_\partial)\to
(V,\pdot,\partial)$ is a non-unital differential graded algebra homomorphism.
\item
        $\widetilde{p_1}\tau^{\qdot}
        \colon (T^c(V),\mu^{\qdot}_\pdot,d^{\qdot}_\partial,1_{\K})\to
        (V,\pdot,\partial,1)$ is a differential graded algebra homomorphism.
\item there exists a (necessarily unique) 2-associative differential algebra homomorphism
        \begin{equation}\label{eq:counit-adjunction}
        \varepsilon_{V}\colon
        ( T^c(V),\mu^{\qdot}_\pdot,\qdot,d^{\qdot}_\partial,1_{\K})\to
        (V,\pdot,\qdot,\partial,1)
        \quad\text{such that }\quad
        (\varepsilon_V)_{|V}=\id_V,
                \end{equation}
        where the multiplication $\qdot$ on $T^c(V)$ is concatenation.
        In fact
$\varepsilon_V=%\widetilde\tau^{\qdot}_1=
        \widetilde{p_1}\tau^{\qdot}$, see Lemma \ref{lem:twistmult}.
\item  the components $m_{i,j}^r\colon V^{\otimes i}\otimes V^{\otimes j}\to V^{\otimes r}$ and $m_n^k\colon V^{\otimes n}\to V^{\otimes k}$ satsify
\begin{equation}\label{eq:associated-binf}
  \sum_{r=1}^{i+j}\qdot^{(r-1)}m^r_{i,j}=\pdot(\qdot^{(i-1)}\otimes\qdot^{(j-1)})
  ,\qquad\qquad
  \sum_{k=1}^n\qdot^{(k-1)}\, m_{n}^{k}=\partial \,\qdot^{(n-1)}
  .
\end{equation}
\end{itemize}

The first of the defining equations \eqref{eq:associated-binf} is just Equation \eqref{eq:mu-r-lift}
of Example \ref{ex:lr-mb}, and so the multibrace structure can be calculated using the recursive formula
\eqref{eq:mij}
(compare \cite[Section 3.3 and Proposition 3.4]{LodayRonco2006}).
The second says
that the $A_\infty$ structure can be calculated,  using \eqref{eq:mnk}, from
\begin{equation}\label{eq:mn}
  m_{n}\;=\;\partial\,\qdot^{(n-1)}\;-\;
  \sum_{k=2}^{n}\qdot^{k-1}\sum_{\myclap{\substack{i,j\geq0\\i+1+j=k}}}
  \id_V^{\otimes i}\otimes m_{n-i-j}\otimes \id_V^{\otimes j}.
\end{equation}
In fact this recursion has a very simple explicit solution
(see \cite[Example 1.3]{markl}) with
\begin{align*}
  m_{2}(v_{1}\otimes v_{2})
  &=\partial(v_{1}\qdot v_{2})
    -\partial(v_{1})\qdot v_{2}
    -(-1)^{\deg{v_{1}}}\; v_{1}\qdot\partial(v_{2}),\\
  m_{3}(v_{1}\otimes v_{2}\otimes v_{3})
  &=\partial(v_{1}\qdot v_{2}\qdot v_{3})
    -\partial(v_{1}\qdot v_{2}) \qdot v_{3}
    -(-1)^{\deg{v_{1}}}\;v_{1}\qdot\partial(v_{2}\qdot v_{3}),\\
  m_{n}(v_{1}\otimes \cdots\otimes v_{n})
  &=\partial(v_{1}\qdot \dots\qdot v_{n})
    -\partial(v_{1}\qdot\cdots\qdot v_{n-1}) \qdot v_{n}
    -(-1)^{\deg{v_{1}}}\;v_{1}\qdot\partial(v_{2}\qdot \cdots\qdot v_{n})\\
&   {\qquad\qquad\qquad\qquad\qquad} +(-1)^{\deg{v_{1}}}\;v_{1}\qdot\partial(v_{2}\qdot \cdots\qdot v_{n-1})\qdot v_{n}
  \text{ for all }n\ge4.\end{align*}

As well as the above properties of the natural projection $T^c(V)\to V$ for any 2-associative differential algebra $V$, we also discuss the natural inclusion $A\to T^c(A)$ for any  $B_\infty$ algebra $A$.

\begin{theorem}\label{thm:incl-hm-b-inf}
    Let $A$ be a  $B_{\infty}$ algebra,  with structure maps given by a multiplication and a differential on the tensor coalgebra $V=T^{c}(A)$ that we denote $\bullet$ and $\partial$ respectively.
    Consider
  \begin{itemize}
    \item
          the 2-associative differential algebra  $(V,\pdot,\qdot,\partial)$,

\noindent          (where the second associative multiplication $\qdot$ on $V$ is concatenation of tensors)
          \item the underlying  $B_{\infty}$ structure  $\mu_\pdot^{\qdot}$ and $d_{\partial}^{\qdot}$ on $V$ given by Theorem \ref{thm:lr-mb}.
  \end{itemize}
  Then
the inclusion map $\iota_1\colon A\to V$ is a homomorphism of $B_{\infty}$ algebras.
\end{theorem}
\begin{proof}
  We will show that the following relations hold,
  \[
    (R_{i,j}^r)\qquad
    m_{i,j}^r(\iota_1^{\otimes i}\otimes \iota_1^{\otimes j})=\iota_1^{\otimes r}\overbar m_{i,j}^r,
    \qquad\qquad\quad
    (S_n^k)\qquad m_n^k\iota_1^{\otimes n}=\iota_1^{\otimes k}\overbar m_n^k,
  \quad\]
where a bar distinguishes the operations on $A$ from those on $V$.
We first observe the concatenation product satisfies $\qdot^{(n-1)}\iota_1^{\otimes n}=\iota_n\colon A^{\otimes n}\to V$.
Applying equations \eqref{eq:associated-binf} to
$\iota_1^{\otimes i}\otimes \iota_1^{\otimes j}$ and to $\iota_1^{\otimes n}$ we get
  \begin{equation*}
    \sum_{r=1}^{i+j}\qdot^{(r-1)}m^r_{i,j}(\iota_1^{\otimes i}\otimes \iota_1^{\otimes j})
    =\pdot(\iota_i\otimes\iota_j)=
    \sum_{r=1}^{i+j}\iota_r\overbar m_{i,j}^r,
    \qquad\quad
    \sum_{k=1}^n\qdot^{(k-1)} m_{n}^{k}\iota_1^{\otimes n}
    =\partial\iota_n=
    \sum_{k=1}^n\iota_k\overbar m_n^k.
\end{equation*}
Next suppose inductively that $(R_{i',j'}^1)$ holds
for $i'\leq i$, $j'\leq j$ except $(i',j')=(i,j)$,
and that  $(S_{n'}^1)$ holds for $n'<n$.
Then $(R_{i,j}^r)$ holds for $r\geq2$ by  \eqref{eq:mu-r},
and $(S_{n}^k)$ holds for $k\geq2$ by \eqref{eq:mnk}. That is, for $r\geq2$ and $k\geq 2$ we have
  \begin{equation*}
    \qdot^{(r-1)}m^r_{i,j}(\iota_1^{\otimes i}\otimes \iota_1^{\otimes j})
    = \qdot^{(r-1)}\iota_1^{\otimes r}\overbar m^r_{i,j}
    = \iota_r\overbar m^r_{i,j},
    \qquad\quad
    \qdot^{(k-1)}m_n^k\iota_1^{\otimes n}
    = \qdot^{(k-1)}\iota_1^{\otimes k} \overbar m_n^k
    = \iota_{k} \overbar m_n^k.
  \end{equation*}
  All terms except those for $r=1$ and $k=1$ in the summations above therefore cancel, leaving the relations $(R_{i,j}^1)$ and $(S_n^1)$ as required.
\end{proof}

\subsection{Universal enveloping 2-associative differential algebras}
\label{subsec:U}

The underlying $B_{\infty}$ algebra functor has a left adjoint,
the \emph{universal enveloping} functor
\[
U:B_{\infty}\text{-alg}\to \dAs^{1,1}\text{-alg}.
\]
It has the following explicit construction. Given a $B_{\infty}$ algebra $A$, consider the free 2-associative differential algebra
$
F_{\dAs^{1,1}}(A)
$,
so that composition with the inclusion
$\alpha:A\to
F_{\dAs^{1,1}}(A)
$
gives a bijection between maps
$\widetilde f\colon
F_{\dAs^{1,1}}(A)
\to V$ in $\dAs^{1,1}$
and linear maps $f\colon A\to V$.
Observe that $\widetilde f\colon
F_{\dAs^{1,1}}(A)
\to V$ will always define a homomorphism of the underlying $B_{\infty}$ algebras, and that $f=\widetilde f\alpha$ defines a $B_{\infty}$ algebra homomorphism if and only if $\widetilde f$ vanishes on elements of the form
\begin{align}
  \alpha(\overbar m_{i,j}(a\otimes b))\!&-\!m_{i,j}(\alpha^{\otimes (i+j)}(a\otimes b),&a\in A^{\otimes i},\,b\in A^{\otimes j},\\
   \alpha(\overbar m_n(c))\!&-\!m_n(\alpha^{\otimes n}(c)),&c\in A^{\otimes n}.\phantom{,a\in A^{\otimes i}}
\end{align}
Here $(m_{i,j}),( m_{n})$ are the associated $B_{\infty}$ structure maps on $
F_{\dAs^{1,1}}(A)
$ and $(\overbar m_{i,j}),(\overbar m_{n})$ those on $A$. Let $I$ be the two-sided ideal of
$
F_{\dAs^{1,1}}(A)
$ generated by the above elements.
Then the universal enveloping 2-associative differential algebra can be defined by the quotient
\[
  U(A)
  \;\;=\;\;F_{\dAs^{1,1}}(A)/ I
\]
since by construction we have a bijection between homomorphisms $\widetilde f\colon
F_{\dAs^{1,1}}(A)
/I\to V$ of 2-associative differential algebras and homomorphisms $f\colon A\to V$ of $B_{\infty}$ algebras.

In view of the results of section \ref{sec:31} there is a simpler presentation of the universal enveloping 2-associative differential algebra:
\begin{proposition}\label{prop:tc=u}
There is a left adjoint to the underlying $B_{\infty}$ algebra functor defined by
\[
T^{c}:B_{\infty}\text{-alg}\to \dAs^{1,1}\text{-alg},
\]
sending a $B_{\infty}$ algebra $A$ to $(T^{c}(A),\partial,\pdot,\qdot)$, where $\partial$ and $\pdot$ are the $B_{\infty}$ structure maps on $A$ and $\qdot$ is concatenation of tensors.

This functor is isomorphic to the universal enveloping algebra $U$.
\end{proposition}
\begin{proof}
Unit and counit maps for the required adjunction between $T^{c}$ and the underlying  $B_\infty$ algebra functor can be defined, by inclusion and multiplication of generators,
\[
  \eta_A=\iota_1\colon A\to T^c(A)\;\in\;B_{\infty}\text{-alg}
  ,\qquad\qquad \varepsilon_V=\widetilde{p_1}\tau^{\qdot}\colon T^c(V)\to V\;\in\;\dAs^{1,1}\text{-alg},
\]
see Theorem~\ref{thm:incl-hm-b-inf} and equation \eqref{eq:counit-adjunction}.
Now the fact that $\widetilde p_{1}\tau^{\qdot}$ is right inverse to both $T^c(\iota_{1})$ and $\iota_{1}$ gives the triangle identities
  \[
\begin{tikzcd}
  T^{c}(A)
  \ar[r,"T^c(\eta_{A})"]\ar[rd,"\id"']
  &T^{c}T^{c}(A)\ar[d,"\varepsilon_{T^{c}(A)}"]
&& V
  \ar[r,"\eta_{V}"]\ar[rd,"\id"']
  & T^{c}(V)\ar[d,"\varepsilon_{V}"]
  \\&T^c(A),&&&V ,
\end{tikzcd}
\]
and we indeed have an adjunction.
The two left adjoints $U$, $T^{c}$ are thus canonically isomorphic:
the required natural isomorphism $\iota_1'\colon U(A)\stackrel\cong\to T^c(A)$ is the 2-associative differential algebra map that corresponds, under the former adjunction, to the unit $\iota_1\colon A\to T^c(A)$ of the latter.
\end{proof}

\subsection{The equivalence of categories}

\begin{definition}[2-associative differential bialgebra]
  A 2-associative differential bialgebra
  is a 2-associative differential  algebra $(V,\partial,\pdot,\qdot,1)$ endowed with a comultiplication $\Delta$ such that
  $(V,\partial,\pdot,\Delta,1)$ is a differential graded bialgebra and $(V,\qdot,\Delta,1)$ is an infinitesimal bialgebra. Together with their homomorphisms they form a category $\dAs^{1,1}\text{-bialg}$.

  A  2-associative differential bialgebra is termed
conilpotent or cofree if it is conilpotent, respectively cofree,  as a unital coalgebra.
\end{definition}

Our aim now is to show there is an equivalence of categories
\[\begin{tikzcd}[every arrow/.append style={shift left=0.7ex}]
 B_{\infty}\text{-alg} \ar[r,"U"]&\ar[l,"\P"]
  \dAs^{1,1}\text{-bialg}_{\text{conil}}\,.
  \end{tikzcd}
\]

By Proposition~\ref{prop:tc=u} the universal enveloping 2-associative differential algebra is just the tensor coalgebra, so the functor in one direction is clear:

\begin{proposition}
The  universal enveloping
2-associative differential algebra
of a $B_\infty$-algebra
is naturally endowed with the structure of a cofree
2-associative differential bialgebra
\[U(A)=(T^c(A),\partial,\pdot,\qdot,\Delta,1_{\K}),\]
where the comutiplication $\Delta$ is given by deconcatenation of tensors, and there is a functor
\[
  U\colon B_\infty\text{-alg}
  \longrightarrow
  \dAs^{1,1}\text{-bialg}.
\]
\end{proposition}
\begin{proof}
  By definition of  $B_\infty$ structures $(T^{c}(A),\partial,\pdot,\Delta,1_{\K})$ is a differential graded bialgebra, and
  $(T^{c}(A),\qdot,\Delta,1_{\K})$ is a (free, cofree) unital infinitesimal
bialgebra by Example \ref{ex:fund-inf-bialg}.\end{proof}

The following result is the promised analogue of the Milnor--Moore theorem, and should be compared with Corollary~\ref{cor:corefl}.

\begin{theorem}\label{thm:milnor-moore}
  The underlying $B_{\infty}$ algebra structure on a 2-associative differential bialgebra restricts to the primitives, and the functor
  \[
  \P\colon \dAs^{1,1}\text{-bialg}
  \longrightarrow
  B_\infty\text{-alg}
\]
is right adjoint to $U$. The unit of the adjuntion is a natural isomorphism
\[
\eta_A\colon A\stackrel\cong\longrightarrow \P(U(A)),
\]
while the counit
\[
\varepsilon_V\colon U(\P(V))\longrightarrow V
\]
is a monomorphism and is an isomorphism if and only if $V$ is conilpotent.

In particular, the functors $U$ and $\P$  define an adjoint equivalence of categories between $B_{\infty}$ algebras and the
full coreflective subcategory of $\dAs^{1,1}\text{-bialg}$
given by the conilpotent objects.
\end{theorem}
\begin{proof}
  If $(V,\partial,\pdot,\qdot,1,\Delta,\varepsilon)$ is a 2-associative differential bialgebra then the conilpotent radical is a 2-associative differential subbialgebra: it is closed under $\partial$, $\pdot$, $\Delta$ by the usual differential bialgebra axioms and is closed under $\qdot$ by Lemma~\ref{lem:inf-delta-prim} (2).
  It therefore follows from the inductive definitions
\eqref{eq:mij} and \eqref{eq:mn}
that it is closed under the underlying $B_{\infty}$ structure maps.

It remains to show that if $(W,\partial,\pdot,\qdot,1,\Delta,\varepsilon)$ is a conilpotent 2-associative differential bialgebra then the underlying $B_{\infty}$ structure retricts to the primitives.
By Proposition \ref{prop:LR-inf-class} we know that $W$ is naturally isomorphic as a unital infinitesimal bialgebra to $T^{fc}(\P(W))$, so we can assume without loss of generality that
$(W,\partial,\pdot)=(T^c(\P(W)),\partial,\pdot)$.
Thus $\P(W)$ has a  $B_\infty$ structure
and $\P(W)\to W$ is the inclusion of a  $B_\infty$ subalgebra
by Theorem~\ref{thm:incl-hm-b-inf}.

Thus, for a general 2-associative differential bialgebra $V$, we have inclusions of $B_{\infty}$ subalgebras $\P(V)=\P(R(V))\to R(V)\to V$, and the $B_{\infty}$ algebra homomorphism $A\to T^{c}(A)=U(A)$ of Theorem~\ref{thm:incl-hm-b-inf} gives a $B_{\infty}$ algebra isomorphism onto its image, \[\eta_{A}\colon A\stackrel\cong\longrightarrow \P(U(A)).\]
The extension of the inclusion $\P(V)\to V$ of the underlying $B_{\infty}$ subalgebra defines the counit \[\varepsilon_V\colon U(\P(V))\longrightarrow V.\]
This factors  as an isomorphism onto its image, the conilpotent radical:
\[\varepsilon_{V\colon }U(\P(V))= T^{c}(\P(R(V)))\stackrel\cong\longrightarrow R(V)\subseteq V.\]
\end{proof}

\begin{corollary}
The categories of $B_\infty$ algebras, of conilpotent 2-associative differential bialgebras and of  cofree 2-associative differential bialgebras
are equivalent.
\end{corollary}
\begin{proof}
  This follows from the theorem, noting that conilpotent implies cofree for $\dAs^{1,1}$ bialgebras since $V\cong T^c(\P(V))$.
\end{proof}
\begin{remark}
The category of cofree differential Hopf algebras $(T^c(A),d,\mu,\Delta)$ has essentially the same objects as the categories above, but more morphisms.
\end{remark}

\end{document}